\documentclass[12pt]{article}

\usepackage[margin=2.5cm]{geometry}
\usepackage{amsfonts}
\usepackage{amssymb}
\usepackage{amsmath}
\usepackage{color}
\usepackage{fancyhdr}
\usepackage{amsthm}
\usepackage{hyperref}
\usepackage{verbatim}
\usepackage{tikz}
\usepackage{comment}

\definecolor{light-gray}{gray}{0.7}
\definecolor{brown}{rgb}{0.59, 0.29, 0.0}

\theoremstyle{plain}
\newtheorem{theorem}{Theorem}

\theoremstyle{definition}
\newtheorem{definition}{Definition}

\newcommand{\seqnum}[1]{\href{http://oeis.org/#1}{\underline{#1}}}
\newcommand{\figmove}{1}
\newcommand{\figarcs}{2}
\newcommand{\figlong}{3}
\newcommand{\figshort}{4}
\newcommand{\figlabels}{5}
\newcommand{\figassociated}{6}
\newcommand{\figposet}{7}
\newcommand{\coveredby}{<\!\!\!\!\cdot}

\title{Brussels Sprouts, Noncrossing Trees, and Parking Functions}
\author{Caleb Ji and James Propp}
\date{}

\begin{document}

\maketitle

\noindent
{\sc Abstract:}
We consider a variant of the game of Brussels Sprouts that,
like Conway's original version, ends in a predetermined number of moves.
We show that the endstates of the game are in natural bijection 
with noncrossing trees and that the game histories are in natural bijection 
with both parking functions and factorizations of a cycle of $S_n$.

\section{Introduction}
Since its creation in the 1960s, 
the two-player topological game Sprouts
introduced by Conway and Paterson~\cite{G}
has been studied for its interesting mathematical properties,
and many basic questions about it remain unresolved.
In contrast, the superficially similar topological game 
{\it Brussels Sprouts} introduced by Conway 
is of no game-theoretic interest at all,
since the identity of the winner is predetermined.
Nonetheless, games that are trivial from the point of view of strategy
may still pose interesting questions for enumerative combinatorics. 
(``If you can't beat 'em, count how many ways they can beat you.'')

In Brussels Sprouts, $n$ crosses are drawn in the plane, 
each with four free ends, or {\it arms}.
To make a move, a player connects two arms with a simple curve 
that does not intersect any of the previously-drawn curves, 
and marks a notch on the curve that provides two further arms,
one on each side of the curve; see page 569 of~\cite{BCG2}.
The game ends when no further moves can be played.  
By interpreting the game as a planar graph and using Euler's formula,
one can show that the game will last precisely $5n-2$ moves.

Variants of Brussels Sprouts have been proposed.
According to \cite{G}, Eric Gans proposed replacing crosses by ``stars''
consisting of $n$ crossbars (that is, stars with $2n$ arms),
and named this new game Belgian Sprouts.
It was quickly realized that there is no need 
to require stars to have an even number of arms;
in the game Cloves (also called Stars-and-Stripes),
stars may have any number of arms.
Here we consider a version of Brussels Sprouts
we call {\it Planted Brussels Sprouts}.  
The initial state consist of a circle with $n$ marked points;
attached to each marked point is an arm interior to the circle,
as in the left side of Figure~\figmove~(the case $n=4$).
To make a move, a player connects two arms with a simple curve
that does not intersect any of the previously-drawn curves,
which we will call an {\it arc},
and marks a notch on the arc that provides two further arms,
one on each side of the arc; see the right side of Figure~\figmove.
(Here, as in some of the other figures,
we leave a slight gap between the arms and the arcs;
when positions become more complicated,
this convention makes it easier to see what has happened.)
The game ends when no more moves are possible;
this always happens after exactly $n-1$ moves 
(see Section 2 for the simple inductive proof).
The game is equivalent to a game of Cloves
in which the initial position is a single $n$-armed star,
though the equivalence involves some topological subtleties
that we discuss in section 5.

We will show that for Planted Brussels Sprouts on $n$ labeled vertices,
the number of topologically distinct endstates is 
${3n-3 \choose n-1}/(2n-1)$ (see entry A001764 in the OEIS),
while the number of topologically distinct lines of game-play is
$n^{n-2}$ (see entry A000272 in the OEIS).
For instance, for $n=3$ there are three topologically distinct endstates
and three topologically distinct lines of play,
while for $n=4$ there are twelve topologically distinct endstates
and sixteen topologically distinct lines of play.
We prove the two formulas by showing that
the endstates are in bijection with noncrossing spanning trees
while lines of play are in bijection with parking functions.

\begin{figure}
\begin{center}
\includegraphics[width=.9\textwidth]{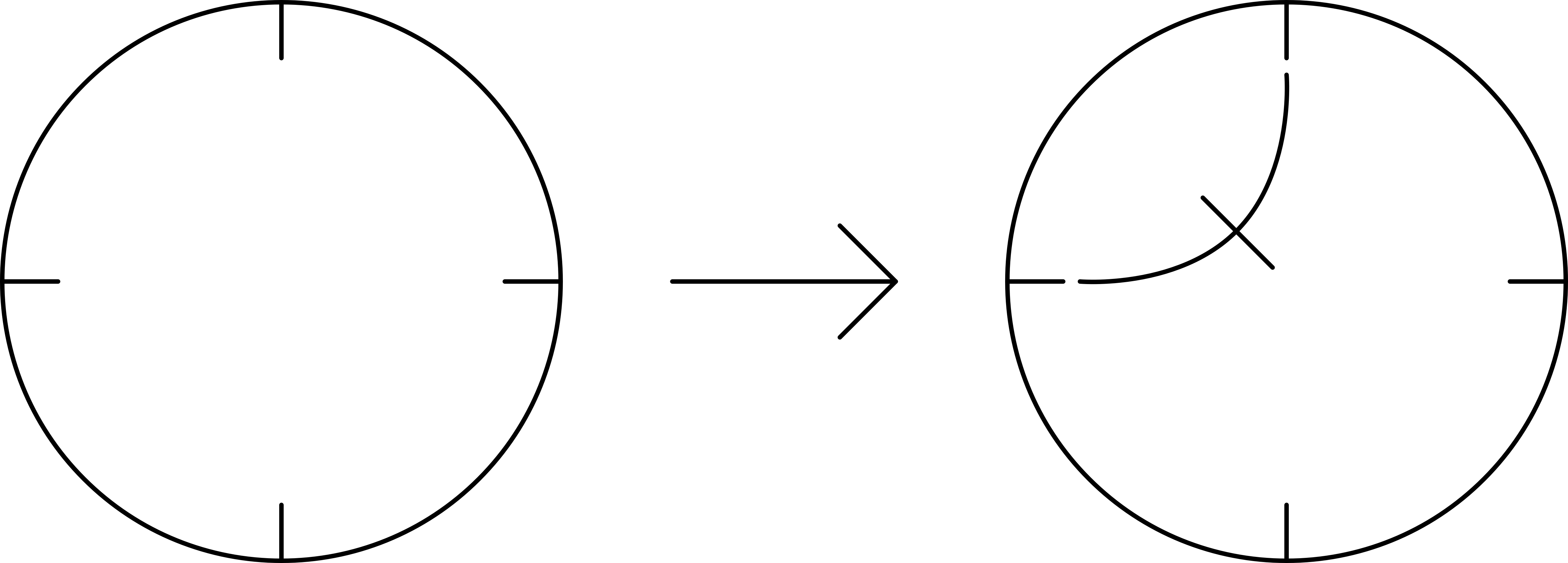}
\end{center}
\caption{A move in Planted Brussels Sprouts.}
\end{figure}

\section{Number of moves}

An elementary induction can be used to show
that the game will last precisely $n-1$ moves.
The claim is trivially true for $n=1$.
For larger $n$, assume that the claim is true for all $n' < n$,
and note that if the first move connects an arm
with another arm $m$ positions away
(we call $m$ the {\it length} of the move),
the game is split into two subgames that are equivalent 
to Planted Brussels Sprouts with $m$ vertices
and Planted Brussels Sprouts with $n-m$ vertices.
Since by the induction hypothesis these subgames
last $m-1$ and $n-m-1$ moves respectively,
the original game lasts $1 + (m-1) + (n-m-1) = n-1$ moves, as claimed.

Since the game lasts $n-1$ moves, $n-1$ arcs are drawn.
These need not correspond to what one might naively call arcs
in the final position of the game. 
For instance, the left panel of Figure~\figarcs, which shows an end position 
for a game of Planted Brussels Sprouts with $n=4$,
contains three arcs, not two or four.
The right panel makes this point more clearly:
the three arcs join $a$ to $b$, $c$ to $d$, and $e$ to $f$ respectively.
Thus the horizontal diameter of the disk
is not itself an arc but does contain two arcs,
and the pairs of arms emanating from it do not count as arcs at all.

It is worth mentioning that the endstate of a game 
does not usually determine the entire history of the game.
For instance, in Figure~\figarcs,
we cannot tell which of arcs $(c,d)$ and $(e,f)$ was drawn first.

\begin{figure}
\begin{center}
\includegraphics[width=.8\textwidth]{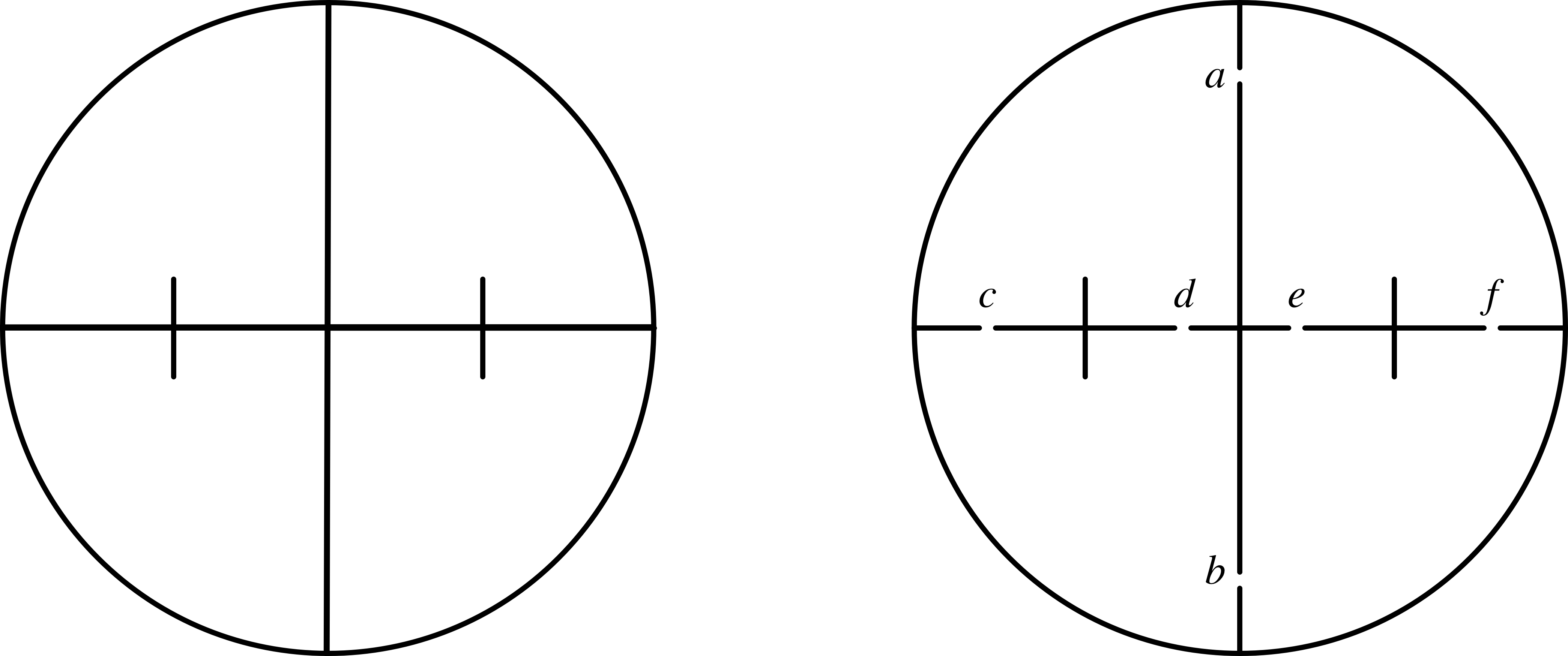}
\end{center}
\caption{An end position and its three arcs $(a,b)$, $(c,d)$, and $(e,f)$.}
\end{figure}

The observation that each state of the game naturally splits up
into one or more smaller subgames,
and that each move separates an existing subgame into two smaller subgames,
will be crucial to studying Planted Brussels Sprouts.
For convenience, we will henceforth assume that
the original arms are labeled $1$ through $n$ in clockwise order.
Arms created later are labeled in an inductive fashion;
if at some stage an arc is drawn connecting
two arms labeled $\alpha$ and $\beta$,
then the two new arms are labeled $(\alpha,\beta)$ and $(\beta,\alpha)$,
as illustrated in Figure~\figlong.
If $P$ (not labeled in the Figure) denotes the point
on the old arc from which the new arms spring,
then surrounding $P$ ones finds, in clockwise order,
the old arm labeled $\alpha$, the new arm labeled $(\alpha,\beta)$,
the old arm labeled $\beta$, and the new arm labeled $(\beta,\alpha)$.
We call these labels the {\it long labels} of the arms.
The utility of these labels stems from the fact that
they contain all topologically relevant information.

It will also be convenient to assign arms {\it short labels},
which are always numbers between 1 and $n$.
When an arc is drawn connecting two arms labeled $i$ and $j$,
the two new arms are labeled $i$ and $j$, as illustrated in Figure~\figshort.
If $P$ again denotes the point on the old arc from which the new arms spring,
then surrounding $P$ ones finds, in clockwise order,
the old arm $i$, the new arm $i$, the old arm $j$, and the new arm $j$.

It is easy to recover the short label of an arm from its long label;
simply take the first component of the pair,
and then take the first component of that pair,
and so on, until one arrives at something
that is a number rather than an ordered pair;
this is the desired short label of the arm.
For instance, an arm with long label $((4,3),(2,(1,5)))$
would have 4 as its short label.

It is also easy to determine the short labels of all the arms
directly from a game state (not necessarily an endstate)
without following an iterative process
that tracks the moves that led to that state.
The arcs that have been drawn at any stage
divide the disk into smaller simply-connected sets (hereafter {\it regions}),
each of which contains a boundary segment of the original disk.
To find the short label of a particular arm $a$
that lies in a particular region,
draw a simple closed clockwise loop that traverses the interior of the region,
starting from $a$ and hugging the boundary as it goes.
Suppose that when the loop first encounters the boundary of the original disk,
the vertex it encounters is the one originally indexed as $i$.
Then $i$ is the short label of $a$. See Figure~\figlabels.

\begin{figure}
\begin{center}
\includegraphics[width=.9\textwidth]{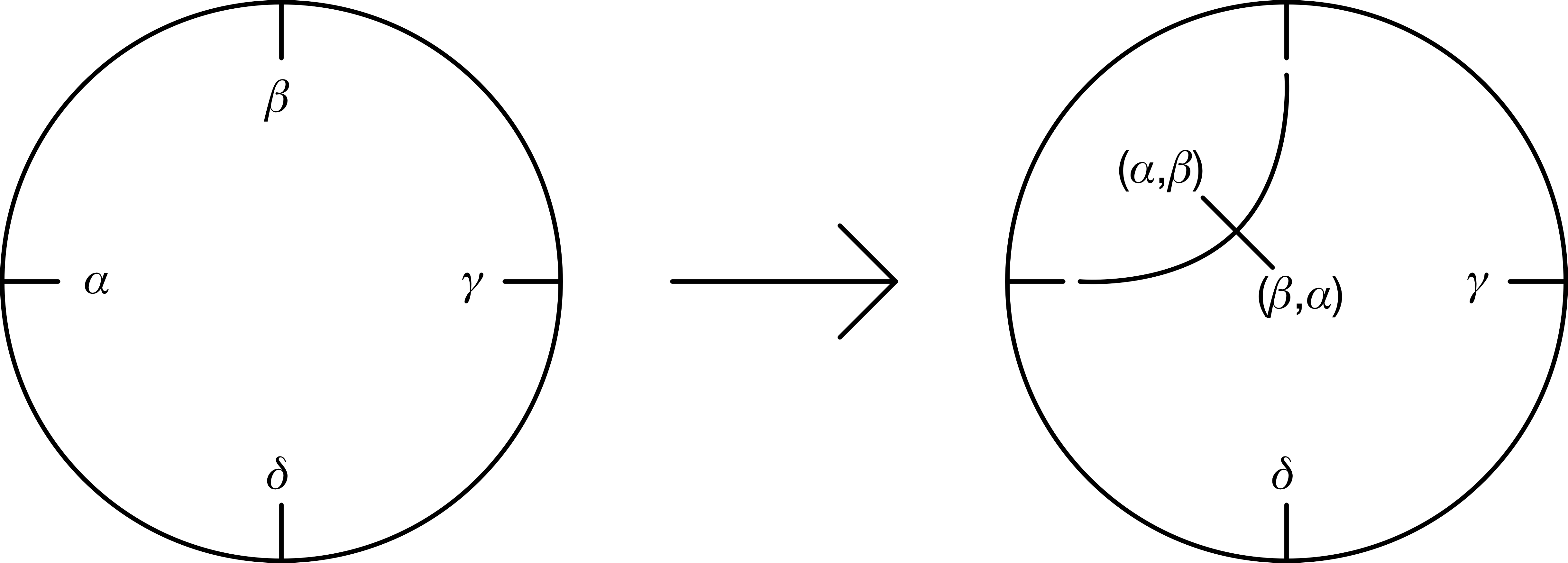}
\end{center}
\caption{Long labels of arms, constructed iteratively.}
\end{figure}

\begin{figure}
\begin{center}
\includegraphics[width=.9\textwidth]{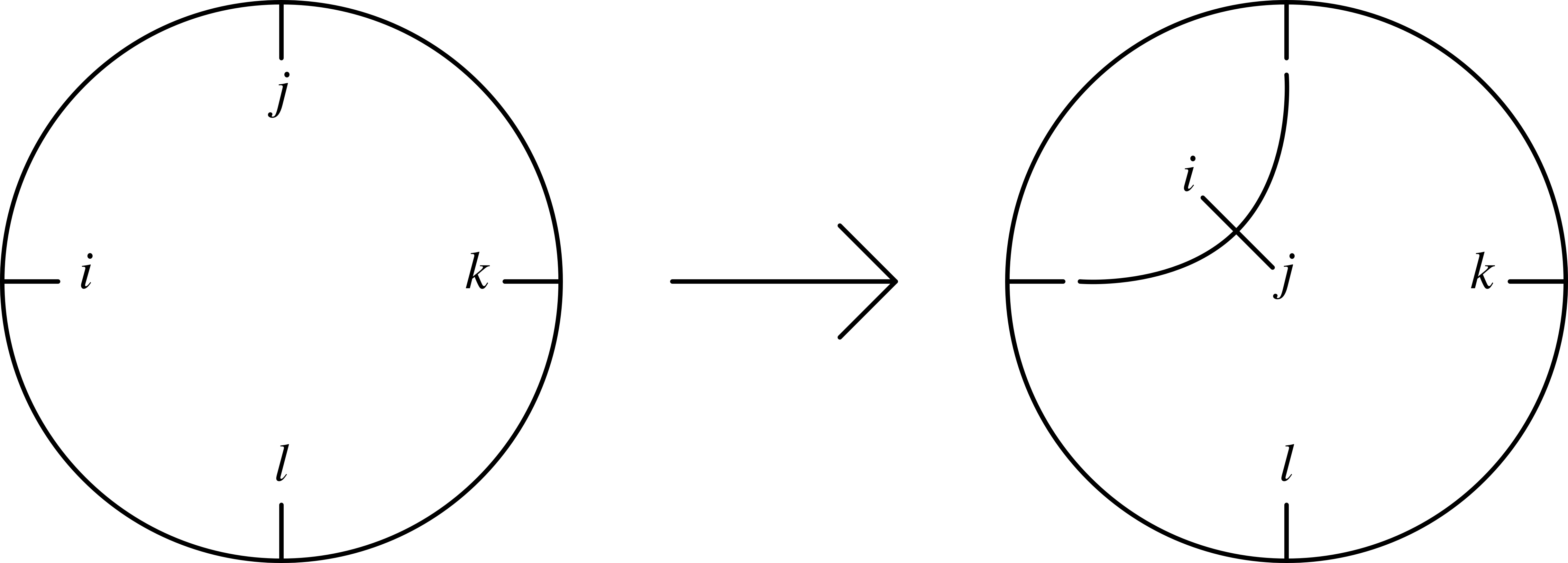}
\end{center}
\caption{Short labels of arms, constructed iteratively.}
\end{figure}

\section{Number of endstates}
\label{end_states}
After $n-1$ moves have been made 
in a game of Planted Brussels Sprouts of order $n$,
an endstate is reached.  
Two endstates are considered topologically equivalent if and only if
there is a homeomorphism of the disk to itself
that carries one picture to the other
and preserves the long labels of the arms.
Note that the long label of an arm contains (as its two components)
the long labels of the two arms that were joined when 
that arm was first created, and so on recursively; 
thus each long label carries its own ``ancestry'' within itself.

Two lines of game-play are considered topologically equivalent if and only 
if for all $1 \leq k \leq n-1$, the state of the first line of play after $k$ moves
is topologically equivalent to the the state of the second line of play after $k$ moves.

Two Planted Brussels Sprouts Games (hereafter {\it PBSG}s)
are topologically equivalent iff their respective moves 
(that is, the arms that are added, listed in chronological order)
have the same long labels.
Also, two Planted Brussels Sprouts Endstates (hereafter {\it PBSE}s)
are topologically equivalent iff they have the same set of long labels.

Each arc that is formed by joining two arms 
during the course of play can be assigned a label,
namely the set whose two elements are
the short labels of the arms being joined
(or, equivalently, the short labels of the
two new arms that get created when that arc is drawn).
For instance, in Figure~\figshort,
the arc that gets added has label $\{i,j\}$;
in the next move of the game, one of three arcs will be added, 
with label $\{j,k\}$, $\{j,l\}$, or $\{k,l\}$.
The labels of the $n-1$ arcs that are drawn in the course of a game
can be thought of as edges
in the complete graph with vertex set $\{1,2,\dots,n\}$.
We picture the $n$ vertices as being arranged
on a circle, labeled 1 through $n$ in clockwise order.
In this way, each state of the game corresponds to a graph $G$.
This graph contains no multiple edges because
when arms labeled $i$ and $j$ are joined,
the new arms that are created that bear the labels $i$ and $j$
will lie on opposite sides of the new arc,
as will any subsequent arms labeled $i$ and $j$,
so that it will never again be possible
to draw an arc labeled $\{i,j\}$.
Furthermore:

\begin{theorem}
\label{NCT}  
The topologically distinct endstates of the game that starts with $n$ labeled arms 
are in bijective correspondence with 
the topologically distinct noncrossing trees on $n$ labeled vertices.
\end{theorem}

\begin{figure}
\begin{center}
\includegraphics[width=.8\textwidth]{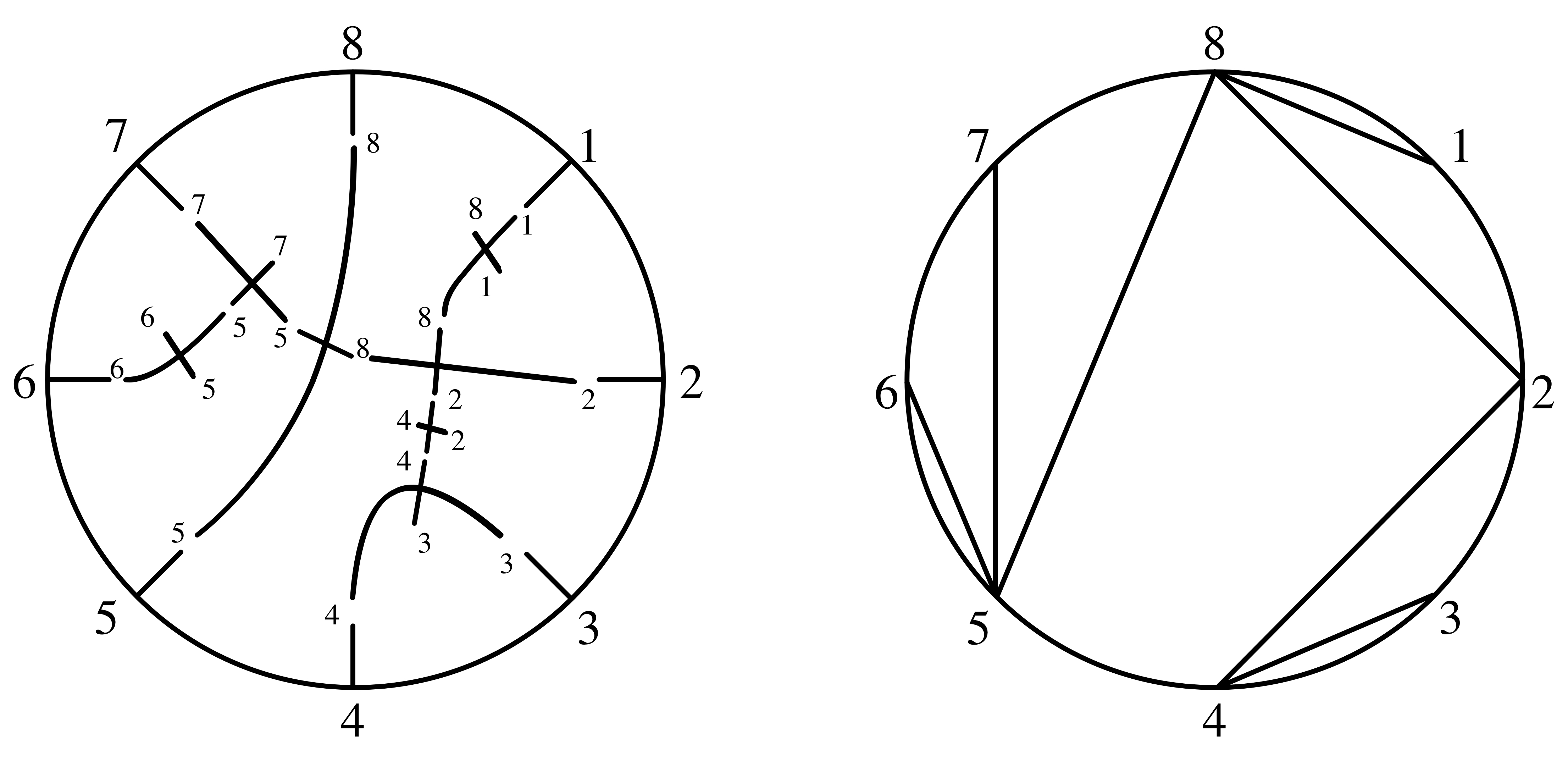}
\end{center}
\caption{A Planted Brussels Sprouts Endstate (left)
and the associated Noncrossing Spanning Tree (right).}
\end{figure}

Before we prove this theorem, 
we will first define a few terms and then give an intuitive picture
of the relationship between PBSEs and NSTs (noncrossing spanning trees).

\begin{definition}
A {\it noncrossing tree} on $n$ labeled vertices is a tree 
on $n$ vertices labeled $1, 2, \ldots, n$ with the property that
when the vertices are written in increasing clockwise order along a circle, 
no two distinct edges of the tree intersect in their interiors.  
\end{definition}

\begin{definition}
A {\it primary arc} is an arc that joins two arms on the original circle.
If an arc joining arms labeled $i$ and $j$ is primary,
then in every playing sequence of the game, 
that arc is drawn prior to all other arcs
involving an arm labeled $i$ or $j$.
On the other hand,
if an arc joining arms labeled $i$ and $j$ is not primary,
then in every playing sequence of the game, 
there will be an earlier-drawn arc
involving an arm labeled $i$ or $j$.
For instance, in the left panel of Figure~\figassociated,
the primary arcs are $\{5,8\}$ and $\{3,4\}$.
\end{definition}

\begin{definition}
In a noncrossing tree on $n$ vertices, 
an edge between vertices $i$ and $j$ where $i<j$ is a {\it primary edge}
iff vertex $i$ is not connected by an edge to any vertex 
in the set $\{j+1, j+2, \ldots, n, 1, 2, \ldots, i-1\}$ 
and vertex $j$ is not connected by an edge to any vertex 
in the set $\{i+1, i+2, \ldots, j-1\}$.
Equivalently, 
an edge $e$ is not primary if and only if one can pivot $e$, 
holding one endpoint fixed and moving the other endpoint clockwise around the circle,
so as to obtain another edge $e'$ in the tree.
For instance, in Figure~\figassociated,
the edge $\{2,4\}$ fails to be primary for two reasons:
it can be pivoted to both $\{2,8\}$ and $\{3,4\}$.
The primary edges in this noncrossing spanning tree 
are $\{5,8\}$ and $\{3,4\}$.
\end{definition}

The motivation for defining primary arcs and edges 
is the fact that they both split up 
the original game or the graph into two disjoint parts.  
The fact that they correspond to each other 
is the basis for the proof of Theorem~\ref{NCT}.

One way to picture the correspondence
is to imagine that the two arms that form a notch
are free to move independently of each other.
Pick some parameter $t$ in $[0,1]$.  When we draw an arc
joining an arm labeled $i$ (attached to a point $p$) 
to an arm labeled $j$ (attached to a point $q$),
forming a curve of length $L$ joining $p$ to $q$,
we take new points $p',q'$ along this curve
at respective distance $(t/2)L$, $(1-t/2)L$
from $p$ along the curve,
and we attach new arms labeled $i$ and $j$ to these points,
so that as we move from $p$ to $q$,
the new arm labeled $i$ is on the left
and the new arm labeled $j$ is on the right.
When $t=1$, we have a PBSE; when $t=0$, we have an NST.
What prevents this attractive picture from being a valid proof
is some hidden details about how this homotopy occurs,
given that the concept of ``arc'' is an inherently dynamical one
whose nature depends on the discrete-time-evolution of the state itself
over the course of the game (recall Figure~\figarcs).
The homotopy can be well-defined,
but we adopt a different approach.

\begin{proof}
First we make an observation: after $k$ arcs have been drawn in a game, 
the game is separated into $k+1$ subgames.  
The sets of arms of these subgames are nonempty and partition the original set of arms, 
and furthermore, these subgames are disjoint in that 
none of the $n-k-1$ remaining arcs to be drawn can connect arms from different subgames.  
These facts are easy to see through considering each arc: 
each arc connects two arms in the same subgame 
and separates that subgame into two disjoint subgames 
whose sets of arms form a partition of the set of arms of the original subgame.

Next we show by (strong) induction 
that the graphs obtained from an endstate are trees.
The base case $n=1$ is clear.
For $n>1$, note that the first arc (joining arms $i$ and $j$, say)
divides the game into two subgames,
and by the induction hypothesis each of these subgames corresponds to a tree,
one containing the vertex $i$ and the other containing the vertex $j$.
The first arc then corresponds to edge $\{i,j\}$,
whose inclusion turns the two-component forest into a tree.

Now we show that every tree obtained from an endstate 
is embedded in the plane as a noncrossing tree.  
Assume for sake of contradiction that there is a crossing of edges.  
Without loss of generality, say it comes from edges between 
vertices $i$ and $k$ and between $j$ and $l$, where $i<j<k<l$, 
and that the arc between $j$ and $l$ is drawn after the arc between $i$ and $k$.  
Consider the first arc that is drawn between 
an arm from the set $\{l+1, l+2, \ldots, n, 1, 2, \ldots, j-1\}$ 
and an arm from the set $\{j+1, j+2, \ldots, l-1\}$.  
(It might be the aforementioned arc between $i$ and $k$
or it might be an earlier-drawn arc.)
We know this occurs before the arc between $j$ and $l$ is drawn, 
and that every arc drawn earlier preserves the fact that 
such an edge would separate $l$ and $j$ into distinct subgames.  
Thus after it is drawn, there is no way for an arc to connect $l$ and $j$, contradiction.
Thus the graph $G$ cannot have any crossing edges and is a noncrossing tree.

Now we show that any two distinct endstates must correspond 
to two different noncrossing trees.  Take a noncrossing tree $T$; 
we will show that it uniquely determines an end-state $A$.  
First, we claim that the primary edges of $T$ 
correspond precisely to the primary arcs of $A$. 
Indeed, take a primary edge connecting vertices $i$ and $j$ with $i<j$ in $T$ 
and assume that it is not a primary arc in $A$.  
This implies that either $i$ or $j$ must have been 
connected to a different arm before being connected to $j$.  
Without loss of generality, say this is true of arm $i$.  
Since the edge between $i$ and $j$ is primary in $T$, 
we know that $i$ could not have been connected 
to any vertex in the set $\{j+1, j+2, \ldots, n, 1, \ldots, i-1\}$.  
Thus it must have been connected to a vertex in the set $\{i+1, \ldots, j-1\}$, 
but this would ensure that it would be placed in a different subgame from arm $j$.  
Then there could have been no arc joining $i$ and $j$ in $A$, contradiction.  
Thus every primary edge in $T$ corresponds to a primary edge in $A$. 
Now take a primary arc of $A$ connecting arms $i$ and $j$.  
Then there must be an edge connecting vertices $i$ and $j$ in $T$.  
Because the arc is primary in $A$, this implies that 
there can be no edges between $i$ 
and the vertices from the set $\{j+1, \ldots, n, 1, \ldots, i-1\}$ 
and no edges between $j$ 
and the vertices from the set $\{i+1, \ldots, j-1\}$.  
Thus, the edge between $i$ and $j$ in $T$ must be a primary edge.  
Thus $A$ has the same primary arcs as the primary edges of $T$.  
Then apply the same argument to each of the subgames of $A$ 
which its primary arcs divide it into.  
Continuing in this fashion, we see that 
the ancestry of each arc in $A$ is uniquely determined by $T$, 
so no two distinct endstates can correspond to the same noncrossing tree.

Finally, we show that any noncrossing tree $T$ can be realized 
from some possible endstate of the game.  We do this by induction.  
The result is easily verified  for the cases of $n=1$, $2$, and $3$.  
Now our goal is to find a primary edge in $T$; 
then by simply drawing the corresponding arc 
and applying the inductive hypothesis 
to the two subgames that the arc creates, 
we obtain the desired result. 
Now take any vertex on the graph, and apply the following algorithm 
until a primary edge is reached.  
Moving counterclockwise from the vertex around the graph, 
take the first vertex it is connected to by an edge.  
If the edge between them is a primary one, stop.  
Otherwise, take that new vertex and repeat the algorithm on it.  
We claim that this algorithm will terminate and return a primary edge.  
Indeed, for each vertex $i$, let $f(i)$ be the first of its neighbors 
in the counterclockwise ordering. 
Since the vertex set is finite, the sequence $i$, $f(i)$, $f(f(i))$, \dots 
must eventually repeat. 
The repeating pattern gives a cycle that contradicts acyclicity of the spanning tree 
unless the repeating pattern repeats with period 2. 
But if $j=f(i)$ and $i=f(j)$, then $\{i,j\}$ is a primary edge.
Thus every noncrossing tree must have a primary edge, finishing this proof.
\end{proof}

Let $a_n$ be the number of endstates of a game that begins with $n$ arms.
Then by Theorem~\ref{NCT}, $a_n$ is the number 
of noncrossing trees on $n$ vertices. The first few values of $a_n$ are:
\[a_1 = 1, a_2 = 1, a_3 = 3, a_4 = 12, a_5 = 55, a_6 = 273.\]

This is \seqnum{A001764} in the OEIS.

\begin{theorem}
The number of endstates of a game of Planted Brussels Sprouts
whose initial state has $n$ arms is
\[a_n = \frac{\binom{3n-3}{n-1}}{2n-1}.\]
\end{theorem}

\begin{proof}
The number of noncrossing trees on $n$ vertices is known to be equal to 
the number of ternary trees on $n-1$ vertices, which is known to be equal to 
$\frac{\binom{3n-3}{n-1}}{2n-1}$ by~\cite{Roy}.
\end{proof}

\section{Number of lines of play}
Let $b_n$ be the number of distinct (complete) playing sequences 
for a game of Planted Brussels Sprouts that begins with $n$ arms.
Two playing sequences are considered different if
for some $i$, the vertices connected at the $i$th stage in the first sequence
are not the vertices connected at the $i$th stage in the second sequence.

\begin{theorem}
The number of lines of play of a game of Planted Brussels Sprouts
whose initial state has $n$ arms is
\[b_n = n^{n-2}.\]
\end{theorem}

\subsection{Proof by recursion}
Say the first move on a game beginning with $n$ arms is of length $i$; 
that is, it separates the game into two subgames of sizes $i$ and $n-i$, 
for some $1\le i\le n-1$.
Then of the remaining $n-2$ moves, 
$i-1$ must be played on the subgame of size $i$, 
and the remaining $n-i-1$ moves are played on the other subgame.  
Thus there are $\binom{n-2}{i-1}$ ways to choose 
which moves will be played on which subgame.  
By definition there are $b_i$ and $b_{n-i}$ ways to play the two games, 
giving a total of $\binom{n-2}{i-1}b_ib_{n-i}$ ways 
to move after the first move.  Now given any fixed $i$, 
note that connecting vertices $a$ and $a+i$ (taken mod $n$) 
as $a$ ranges from $1$ through $n$ are the ways 
to make a first move of length $i$.  
The associated sum $\sum_{i=1}^{n-1}\binom{n-2}{i-1}b_ib_{n-i}$
undercounts by a factor of $n$,
corresponding to the $n$ different choices of an initial arm,
but also overcounts by a factor of 2,
since either of the two ends of the first arc
could be treated as the initial arm
(moves of length $n-i$ are equivalent to moves of length $i$).  
Thus the sum needs to be multiplied by $n/2$, giving the recursion
\begin{equation}
\label{rec}
b_n=\dfrac{n}{2}\sum_{i=1}^{n-1}\dbinom{n-2}{i-1} b_i b_{n-i}.
\end{equation}

It is well-known that the number of trees on $n$ unlabeled vertices, $T_n$, 
is $n^{n-2}$.  It is easy to check that 
the first few values of $b_n$ correspond with this as well.  
Thus, if we show that $T_n$ follows the same recursion as that in Equation~\ref{rec}, 
the theorem is proven. \\

Consider all spanning trees on vertices $1, 2, \ldots, n$, 
and note that proportionately exactly 
$\frac{n-1}{\binom{n}{2}}=\frac{2}{n}$ of them 
have an edge between vertices $1$ and $2$.  
Indeed, each spanning tree has $n-1$ edges 
out of $\binom{n}{2}$ total possible edges, 
and when all these trees are considered, 
no edge occurs more than any other.  
Now we count the number of spanning trees 
where vertices $1$ and $2$ are connected by an edge.  
Note that removing this edge results in two subtrees.  
If there are $i-1$ other vertices in the subtree with vertex $1$, 
then there are $n-i-1$ vertices in the subtree with the vertex $2$.  
Since $i$ can range from $1$ through $n-1$ 
and any $i-1$ of the remaining $n-2$ vertices 
can be chosen to be in the first subtree, 
we obtain the recursion 
\begin{equation}
T_n=\dfrac{n}{2}\sum_{i=1}^{n-1}\dbinom{n-2}{i-1}T_iT_{n-i}.
\end{equation}
which is the same recursion as that in Equation~\ref{rec}.  
Thus the number of ways to play Planted Brussels Sprouts 
is the same as the number of spanning trees 
on $n$ distinguishable vertices, as desired. \\

\subsection{Bijection to parking functions}

It is known that the number of parking functions of length $n$ 
is $(n+1)^{n-1}$~\cite{KW}.  
We will now exhibit a bijection between the playing sequences of length $n$ 
and the parking functions of length $n-1$.  
This bijection between maximal chains of noncrossing partitions and parking functions
is due to Stanley~\cite{Stanley};
here, we restate it in terms of playing sequences 
of Planted Brussels Sprouts and prove it. \\

First, recall that a parking sequence of length $n-1$ 
is an ordered $n-1$-tuple $(a_1, a_2, \ldots, a_{n-1})$ of integers 
between $1$ and $n-1$, inclusive, satisfying the property that 
if the $a_i$'s are arranged in increasing order $a'_1\le a'_2\le \cdots\le a'_{n-1}$, 
then $a'_i\le i$ for all $1\le i\le n-1$. \\

Using the same clockwise shift of numbering as used in Section~\ref{end_states}, 
we can interpret each move of Planted Brussels Sprouts 
as dividing the game into two subgames that partition the set of vertices.  
Thus, a move is essentially equivalent to separating 
a contiguous set of arms from the rest of the arms.  
For each of the $n-1$ moves, we obtain a number in the following way.  
Say that a move connects arms $i$ and $j$.  
Now in the subgame in which this move was made, 
consider the labels of the arms immediately counterclockwise 
to arm $i$ and to arm $j$.  Take the lesser one of these labels.

\begin{theorem}
The sequences of $n-1$ numbers obtained in this way 
are distinct for distinct playing sequences, 
and are the parking functions of length $n-1$.
\end{theorem}

\begin{proof}
First we will show that the sequences 
generated in this way are parking functions.  
Indeed, assume the contrary.  
This implies that for some $k$, at least $k+1$ of the numbers are at least $n-k+1$.  
We proceed by induction on $k$.  
The result is easy to verify for the base case of $k=0$.  
Now for the inductive step, note that for a move to generate such a number, 
it must connect to an arm with label at least $n-k+1$.  
Consider the first such move.  
At this point, at most one of the generated numbers is at least $n-k$.  
Furthermore, after this move, these largest $k$ arms 
are split into $t$ separate subgames with $t\ge 2$.  
Because they must be the largest arms in their respective subgames, 
by the inductive hypothesis, at most $1+k+1-t\le k$ of the numbers 
can be one of those $k+1$ largest numbers, contradiction.     \\

Now we will show that each parking function $(a_1, a_2, \ldots, a_{n-1})$ 
can be obtained through exactly one playing sequence, 
which will complete the proof.  We do this by induction.  
For the base cases of $n=2$ and $n=3$, the result is easily checked.  
Now assuming the result for labels $2$, $3$, $\ldots$, $n-1$, 
we prove the result for $n$.   
Since the first term of the sequence is $a_1$, 
say that the first move connects arms $a_1+1$ and $j$ with $j > a_1+1$ or $j=1$.  
This divides the game into a subgame containing 
the vertices $a_1+1$, $a_1+2$, $\ldots$, $j-1$ and another subgame containing 
$j$, $j+1$, $\ldots$, $n$, $1$, $2$, $\ldots$, $a_1$.
Now consider the remaining terms of the sequence: 
$a_2$, $a_3$, $\ldots$, $a_{n-1}$.  
The terms of this subsequence which correspond 
to arms in the first subgame must form a parking sequence for the first subgame, 
and the rest must form a parking sequence for the second subgame.  
This is because they must be played in their respective subgames, 
and each playing sequence of a game must give rise to a parking sequence. 
Thus we must show that such a $j$ strictly 
between $a_1-1$ and $n$ exists and is unique. \\

Let $j$ move clockwise from $a_1+2$ through $n$ and $1$ 
until the following condition is satisfied: 
there are precisely $j-a_1-2$ terms 
of the sequence $a_2$, $a_3$, $\ldots$, $a_{n-1}$ 
which are equal to one of the numbers $a_1+1$, $a_1+2$, $\ldots$, $j-1$, 
and if these terms are arranged in increasing order, 
the $k^\text{th}$ term is at most $a_1+k$.  
In other words, the remaining terms of the sequence 
which are within this subgame form a parking sequence for this subgame.  
Indeed, as $j$ ranges, this condition must occur at some point.  
If $a_1+1$ is not a term in this sequence, then setting $j=a_1+2$ works.  
Otherwise, there is at least $1$ term of $a_1+1$ in the sequence.  
Now as $j$ ranges from $a_1+2$ through $1$, 
define $f(j)$ to be $j-a_1-2$ subtracted from the number of terms 
between $a_1+1$ and $j-1$, inclusive (when $j=1$, consider it equal to $n+1$).  
Thus, $f(a_1+2)$ is equal to the number of terms 
equal to $a_1+1$ in the sequence $a_2$, $a_3$, $\ldots$, $a_{n-1}$.  
Because the original sequence is a parking sequence, 
there are at most $n-a_1-1$ terms at least $a_1+1$.  
Thus, $f(1) \le n-a_1-1 - (n-1-a_1) = 0$.  
Note that $f(j)$ decreases by at most $1$ on each step.  
Now take the first $j$ as $j$ moves clockwise from $a_1+2$ through $1$ 
such that $f(j) = 0$.  
Because there are the right number of terms for a parking sequence 
and because of the minimality of $j$, this choice of $j$ works as desired. \\

Now we show that when this condition is satisfied, 
the remaining terms which fall in the other subgame 
form a parking sequence on that subgame as well.  
Essentially, certain terms which form a parking sequence 
have been taken away from the parking sequence $a_1$, $a_2$, $\ldots$, $a_{n-1}$, 
and we must show that the result is a parking sequence on the remaining terms: 
$1$, $2$, $\ldots$, $a_1$, $j$, $j+1$, $\ldots$, $n$.  
This follows from the nature of parking sequences.  
Indeed, when these are arranged in increasing order, 
for any $i$ with $1\le i\le a_1$, the $i^\text{th}$ term 
is at most $i$ because the original sequence is a parking sequence.  
Next, for any $i$ with $1\le i\le n-j$, 
no more than $i$ of them can be at least $n-i$, for the same reason.  
Thus when the terms form a parking sequence on the first subgame, 
the remaining terms must form a parking sequence on the second subgame. \\

Now it remains to show that such a choice of $j$ is unique.  
Indeed, suppose that there are two choices $j$, $j'$ with $j< j'$.  
Note that because the terms form a parking sequence 
on the arms $a_1+1$, $a_1+2$, $\ldots$, $j-1$, 
then $j-1$ cannot be one of the terms.  
But if the terms also form a parking sequence 
on $a_1+1$, $a_1+2$, $\ldots$, $j'-1$, 
then there are $j'-a_1$ terms in that range, 
but only $j-a_1$ terms in $a_1-1$, $a_1$, $\ldots$, $j-1$.  
This means that there are $j'-j$ terms 
in the range $j$, $j+1$, $\ldots$, $j'-2$, contradiction.

\end{proof}

\subsection{Bijection to factorizations of the cycle}
The $n$-cycle denoted by $(1,2,\cdots,n)$ (in cycle notation)
can be expressed as a product of $n-1$ (but no fewer) transpositions,
and it has long been known~\cite{D} that the number of such expressions 
is equal to $n^{n-2}$.

Recall that in our previous bijective proof, 
each arc gave rise to two labels, and by picking the smaller one, 
we derived a parking function.  
If instead we take those same two labels to be a transposition, 
we claim that the sequences of transpositions that can be obtained this way 
form the $n^{n-2}$ factorizations of the cycle $(1,2,\cdots,n)$.

For example, in the case $n=3$,
there are three lines of play:
the first has arc $\{1,2\}$ followed by arc $\{2,3\}$;
the second has arc $\{1,3\}$ followed by arc $\{1,2\}$;
and the third has arc $\{2,3\}$ followed by arc $\{1,3\}$.
These three lines of play correspond to the factorizations
(12)(23), (13)(12), and (23)(13), respectively,
where the successive inserted arcs
correspond to the transpositions read from left to right
(even though as usual we imagine the permutations
being composed from right to left).
As we will see later, for purposes of the inductive proof
it will be more convenient to increment shift the indices by 1 mod $n$,
so that these three lines of play correspond instead to
the factorizations (23)(13), (12)(23), and (13)(12), respectively.)

\begin{theorem}
For each arc in a PBSG connecting labels $i$ and $j$, 
let $i'$ and $j'$ be the labels immediately counterclockwise from $i$ and $j$ 
in the subgames the arc divides them into, respectively.  
Then consider the sequences of $n-1$ transpositions defined 
by transposing the labels $i'$ and $j'$ obtained from each arc of a PBSG in this way.
Then these sequences of transpositions obtained from the distinct PBSGs, 
applied to the identity permutation in the order they 
are obtained, are the $n^{n-2}$ factorizations 
of the cycle $(1,2,\cdots,n)$ into $n-1$ transpositions.

\begin{proof}
To demonstrate this bijection, we prove three things: 
first, each sequence of transpositions obtained in this way 
is a factorization of $(1, 2, \ldots, n)$; 
second, distinct lines of play correspond to different sequences of transpositions; 
and third, each of the $n^{n-2}$ factorizations can be achieved in this way.

We prove the first claim through induction.  
For the base case $n=2$, the result is clear.  
Now assume the result holds for $n=1, 2, \cdots, k-1$ for some $k\ge 3$.  
We prove the result for $n=k$.  
Say the first arc connects arms $i$ and $j$ with $i<j$.  
This corresponds to the transposition $(i-1, j-1)$, and breaks the game into 
subgames with arms $i, i+1, \ldots, j-1$, and $j, j+1, \ldots, k, 1, \ldots, i-1$.  
Then by the inductive hypothesis, the remaining $n-2$ transpositions 
will act as a cycle on those two subgames.  
Since those cycles are disjoint, the final permutation obtained is 
$(i, i+1, \ldots, j-1)(j, j+1, \ldots, n, 1, \ldots, i-1)(i-1, j-1)=(1, 2, \ldots, n)$, 
as desired.

The second claim follows from the fact that each stage of the game, 
each transposition uniquely determines the arc necessary to obtain it.  
To prove this fact, we use induction.  That is, we induct on the assertion 
that there cannot be multiple distinct ways to produce the same sequence 
of the first $l$ transpositions of a sequence of $n-1$ transpositions, 
for any $1\le l\le n-1$.  Say we are given a sequence of transpositions, 
denoted $t_1, t_2, \ldots, t_{n-1}$, where $t_k$ 
transposes $i_k$ and $j_k$ for $1\le k\le n-1$.  
Then for $t_1$ to be obtained, the first move must connect $i_1+1$ and $j_1+1$, 
so the result holds for the base case.  
Now assume that there is a unique set of $l$ moves to obtain 
the transpositions $t_1, t_2, \ldots, t_l$ for some $1\le l \le n-2$.  
Indeed, if there is no way to obtain those transpositions, 
then there trivially cannot be multiple ways 
to obtain the first $l+1$ transpositions.  
Then in order for $t_{l+1}$ to be obtained, 
both $i_{l+1}$ and $j_{l+1}$ need to be in the same subgame, 
and the only move will be to connect the labels positioned 
directly clockwise to them in that subgame.  
Thus there cannot be multiple ways to achieve $t_1, t_2, \ldots, t_{l+1}$, as desired.

To prove the third claim, 
first consider what defines an achievable sequence of transpositions from a PBSG.  
The only restriction is that each transposition act on two elements of the same subgame; 
in particular, after drawing an arc between $i+1$ and $j+1$, 
the remainder of the transpositions must act on elements 
of either $(i+1, i+2, \ldots, j)$ or $(j+1, j+2, \ldots, n, 1, \ldots, i)$.

Now take any sequence of transpositions $t_i$, $1\le i \le n-1$ 
such that $t_{n-1}t_{n-2}\cdots t_1=(1, 2, \ldots, n)$.  
Then we have $(1, 2, \ldots, n)t_1t_2\ldots t_{n-1} = e$.  
Note that as these transpositions are successively applied to $(1, 2, \ldots, n)$,  
the cycle type of the permutation goes from $1$ cycle to $n-1$ cycles.  
Furthermore, each individual transposition 
can only increase the number of cycles by $1$, 
and that occurs if it transposes elements from different cycles.  
Thus each of the transpositions must have this property.  
That is, each successive transposition $t_{k}$ must add one cycle 
to the number of cycles in the permutation 
defined by $(1, 2, \ldots, n)t_1\ldots t_{k-1}$.

Say $t_1 = (i, j)$.  Then we have 
$$(1, 2, \ldots, n)t_1=(1, 2, \ldots, n)(i,j)=
(j+1, j+2, \ldots, n, 1, \ldots i-1, i)(i+1, i+2, \ldots, j-1, j).$$
Then note that in order for $t_2$ to add to the number of cycles, 
it must transpose two arms of the same cycle, 
for if it doesn't, then it will merge the two cycles it connects.  
This is true at each stage; $t_k$ must act on two elements 
of the same cycle of $(1, 2, \ldots, n)t_1\ldots t_{k-1}$.  
But note that at each stage, the cycles correspond precisely 
to the subgames which drawing the appropriate arc divides the game into.  
Indeed,  when a permutation is applied to two elements of the same cycle, 
the resulting two cycles correspond to the same two subgames that 
drawing the corresponding arc divides the subgame it is drawn in into.  
Thus each of these sequences of transpositions is allowed, 
because they act within each subgame.  
Thus each of these sequences can be obtained through a PBSG, completing the proof.

\end{proof}
\end{theorem}

\section{Odds and ends}

\subsection{Variants}
A seemingly slightly different variant of Brussels Sprouts
game that yields the same two integer sequences A001764 and A000272 
starts from an $n$-pointed star on a sphere.
This is in fact the same game in disguise.
Indeed, we may replace the star by a small circle with $n$ arms protruding outward, 
without changing any aspects of the legal moves.
The circle partitions the surface of the sphere into two parts, 
one of which is the region which the arcs will be drawn on.  
Since this region is homeomorphic to the open disk, 
the enumerative properties of this variant 
are the same as those of the original game.

We might also consider an $n$-pointed star on the plane.
It is not hard to show that the number of endstates for this game 
is exactly $n$ times the number of endstates for the game
treated in the previous paragraph.
To see this, think of the plane as a punctured sphere.
When playing on a sphere, the drawn arcs 
always divide the sphere into $n$ regions at the end of the game.  
We may choose to place the puncture in any one of these regions, 
and then by performing a stereographic projection onto the plane, 
we recover $n$ distinct endstates for each endstate of the game on a sphere.
For the same reason,
the number of lines of play for the star-on-a-plane game
is exactly $n$ times the number of lines of play 
for the star-on-a-sphere game.

On the other hand, we might consider initial states
consisting of several stars (on a plane or a sphere),
as in the original version of Brussels Sprouts.
Now it might appear that there are infinitely many endstates,
and indeed infinitely many initial moves,
since an arc joining an arm of one star
to an arm of another star
has the freedom of winding around one star or the other;
if we imagine the stars as being fixed in place (on the plane or sphere),
then there is no homotopy that can undo the winding.
Hence it is unclear whether there are enumerative results to be obtained here
(though one might try to sweep the winding issue under the rug
by paying attention only to the combinatorics, and not the topology,
of the connection patterns of the arcs).

Finally, one might consider a ``type B'' version of Planted Brussels Sprouts
that maintains 180-degree rotational symmetry throughout the game.
This requires that there be two kinds of moves:
a single arc joining two points 180 degrees apart,
and a pair of matched arcs related by 180 degree rotation.

\subsection{Lines of play leading to a particular endstate}

\begin{figure}
\begin{center}
\includegraphics[width=.9\textwidth]{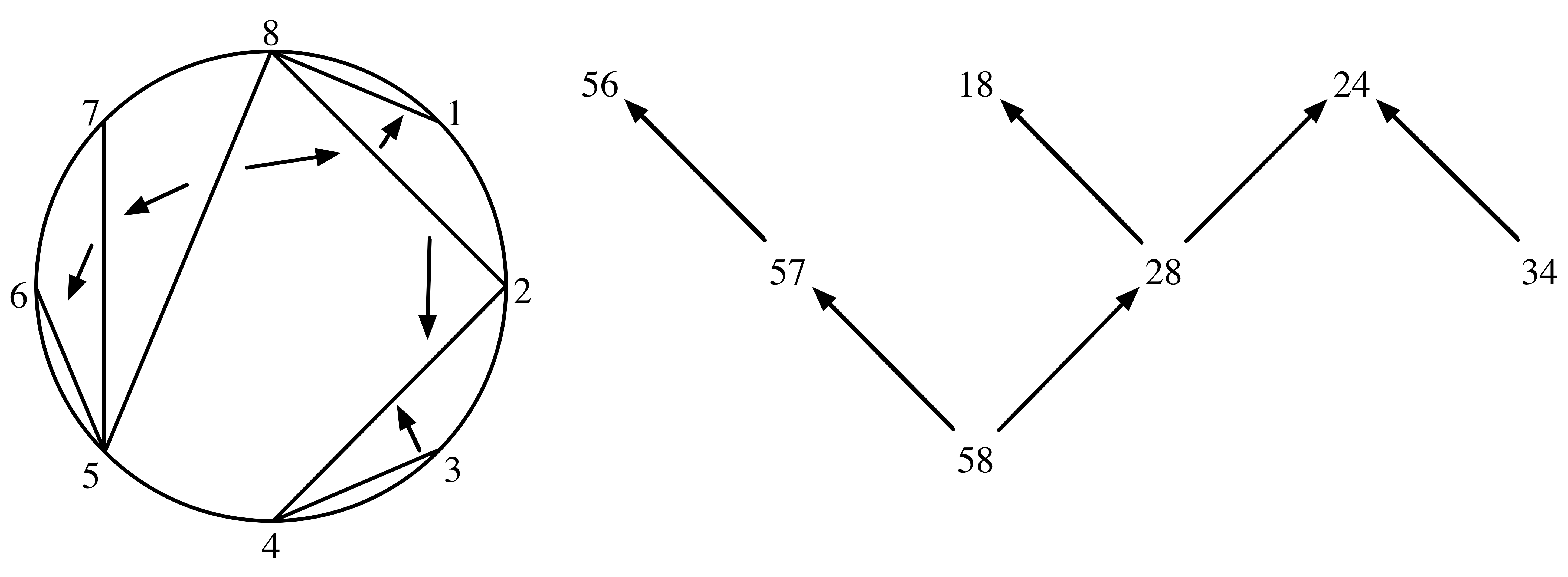}
\end{center}
\caption{The partially ordered set associated with a noncrossing spanning tree. 
Each covering relation $e$ $\coveredby$ $e'$ is indicated by 
an arrow $e \rightarrow e'$, 
indicating that edge $e$ must be played before edge $e'$.}
\end{figure}

It is not hard to show that, for each Planted Brussels Sprouts Endstate, 
the PBSGs having that particular PBSE as their endstate
correspond to the linear extensions of a particular poset associated with the PBSE.
We can most conveniently describe the poset using the NST picture.
Given a fixed NST $T$ with edge-set $E_T$,
and given two edges $e,e' \in E_T$,
write $e \rightarrow e'$ to mean that $e'$ can be obtained from $e$ 
by swinging $e$ counterclockwise around one of its two endpoints.
We have already shown that the digraph associated with this relation is acyclic; 
let $\leq_T$ be its reflexive-transitive closure.
Then we claim that the PBSGs compatible with $T$
are precisely the linear extensions of the poset $(E_T, \leq_T)$.
This is easily proved by induction making use of the primary edges,
which are the minimum elements of the poset.
See Figure~\figposet.

\bigskip

{\sc Acknowledgments}: This article grew out of a problem
originally presented by the second author
at a meeting of the Cambridge Combinatorics and Coffee Club.
CCCC was founded by Richard Stanley, 
and the meeting was hosted by the Worldwide Center of Mathematics.
Alex Postnikov devised the proof-by-recursion of Theorem 3.
Subsequently, Tanya Khovanova brought the problem
to the attention of the first author
(as part of the PRIMES program), who found a bijective proof of Theorem 3, 
and also conjectured and proved Theorem 2.
Nathan Williams found his own proof of Theorem 5.
We are grateful to the anonymous referee
who made many suggestions that improved this article.

\bibliography{JiPropp4}{}

\begin{thebibliography}{1}

\bibitem{BCG2}
E.~R. Berlekamp, J.~H. Conway, and R.~K. Guy, {\em Winning Ways for Your
  Mathematical Plays, Volume 2}, A~K~Peters, Ltd., Reading, Massachusetts,
  2003.

\bibitem{D}
J.\ D\'enes, The representation of a permutation as the product of a minimal
  number of transpositions and its connection with the theory of graphs, {\em
  Publ.\ Math.\ Institute Hung.} {\bf 4} (1959), 63--70.

\bibitem{G}
M.\ Gardner, Mathematical games: Of sprouts and brussels sprouts, games with a
  topological flavor, {\em Scientific American} {\bf 217} (1967), 112--115.
\newblock Appears as Chapter 1 in {\em Mathematical Carnival}, 1989.

\bibitem{KW}
A.~G. Konheim and B.~Weiss, An occupancy discipline and applications, {\em SIAM
  J.\ Applied Math.} {\bf 14} (1966), 1266--1274.

\bibitem{Roy}
M.\ Roy, Enumeration of noncrossing trees on a circle, {\em Discrete
  Mathematics} {\bf 180} (1998), 301--313.

\bibitem{Stanley}
R.\ Stanley, Parking functions and noncrossing partitions, {\em The Electronic
  Journal of Combinatorics} {\bf 4} (1997), R20.

\end{thebibliography}
\bibliographystyle{jis}

\end{document}